\documentclass[11pt,a4paper]{amsart}

\usepackage{fullpage}
\usepackage{amsmath}
\usepackage{amsfonts}
\usepackage{amssymb}
\usepackage{amsthm}
\usepackage{mathtools}
\usepackage[utf8]{inputenc}
\usepackage{enumitem}
\usepackage{dsfont} 
\usepackage{breqn}
\usepackage{afterpage}
\usepackage{longtable}
\usepackage{indentfirst}
\usepackage{caption}
\usepackage{subcaption}
\usepackage{graphicx}
\graphicspath{ {./images/} }
\usepackage{tikz-cd}

\newtheorem{theorem}{Theorem}[section]
\newtheorem*{thm}{Main Theorem}
\newtheorem{lemma}[theorem]{Lemma}

\newtheorem{proposition}[theorem]{Proposition}

\theoremstyle{definition}
\newtheorem{definition}[theorem]{Definition}

\newtheorem{theorem-definition}[theorem]{Theorem-Definition}

\theoremstyle{remark}
\newtheorem{remark}[theorem]{Remark}

\numberwithin{equation}{section}

\usepackage{color}
\usepackage{xcolor}

\newcommand{\lam}{\lambda}

\newcommand{\C}{\mathbb{C}}
\newcommand{\R}{\mathbb{R}}

\renewcommand{\P}{\mathbb{P}}
\newcommand{\A}{\mathcal{A}}

\renewcommand{\O}{\mathcal{O}}

\renewcommand{\L}{\mathcal{L}}
\newcommand{\J}{\mathcal{J}}
\newcommand{\F}{\mathcal{F}}
\newcommand{\M}{\mathcal{M}}

\DeclareMathOperator{\supp}{Supp}

\renewcommand{\epsilon}{\varepsilon}

\begin{document}

\title[]{Propagation of equilibrium states in stable families of endomorphisms of $\P^k(\mathbb C)$.}

\author{Maxence Brevard}
\author{Karim Rakhimov}

\address{Université de Toulouse - IMT, UMR CNRS 5219, 31062 Toulouse Cedex, France}
\email{maxence.brevard@univ-toulouse.fr, mbrevard@hotmail.fr}
\address{V.I. Romanovskiy Institute of Mathematics of Uzbek Academy of Sciences,  Tashkent, Uzbekistan}

\curraddr{}

\email{karimjon1705@gmail.com}

\subjclass[2010]{}
\date{}


\begin{abstract}

We prove that, within any holomorphic  family of endomorphisms of $\P^k(\mathbb C)$ in any dimension $k\geq 1$ and algebraic degree $d\geq 2$, the measurable holomorphic motion associated to dynamical stability in the sense of Berteloot-Bianchi-Dupont
preserves the class of equilibrium states
associated with weight functions $\psi$ satisfying $\sup\psi - \inf \psi < \log d$.

\end{abstract}

\maketitle

\section{Introduction}\label{s:intro}

Let $f : X \to X$ be a continuous map on a  compact metric space $X$.
Given a Borel bounded function $\psi: X\to \mathbb R$, an \emph{equilibrium state} with respect to the \emph{weight} $\psi$ is an $f$-invariant probability measure $\nu$ on $X$ maximizing the pressure functional
$P_\nu(\psi) := h_\nu(f) + \int \psi \nu$, where $h : \nu \mapsto h_\nu(f)$ is the measure-theoretic entropy. The existence of equilibrium states follows from a compactness argument, whereas their uniqueness is a strong statistical property. In this paper, we are interested in the persistence of equilibrium states and their uniqueness in \emph{stable families of endomorphisms of $\P^k=\P^k (\C)$}, which we now define.\\

Given a (connected)
complex manifold $M$, a holomorphic family
$\{f_\lam : \P^k \to \P^k\}_{\lam \in M}$
of endomorphisms of $\mathbb P^k$ parametrized by $M$ is
a holomorphic map $F : M\times \P^k \to M \times \P^k$ of the form  $F(\lam,z) = (\lam, f_\lam(z))$  such that all the maps $f_\lam$
are holomorphic endomorphisms of $\P^k$ of the same algebraic degree $d \ge 2$.
In dimension $k=1$, the notion of stability within such families was introduced by Ma{\~n}\'e-Sad-Sullivan \cite{MSS83}
and Lyubich \cite{Ly83a}. A very fruitful 
characterization based on potential theory
was later introduced by DeMarco \cite{dM03}, see also  \cite{Prz85,Si81}.
The family $\{f_\lam : \P^1 \to \P^1\}_{\lam \in M}$  is said to be \emph{stable} if  there exists a \emph{holomorphic motion} of the Julia sets
(i.e., the supports of the measures of maximal entropy of $f_\lam$ \cite{Ly83b})
near every $\lam_0 \in M$, namely  there exists a neighbourhood $U$ of  $\lambda_0$ and,
for every $\lam \in U$, an injective  map 
$\Phi_{\lam_0,\lam} : J_{\lam_0} \to J_\lam$
depending 
holomorphically on $\lam$. By Hurwitz theorem,
these maps are automatically continuous in $z\in J_{\lam_0}$, hence they
 define a topological conjugacy between the systems $(J_{\lam_0}, f_{\lam_0})$ and $(J_\lam, f_\lam)$. 
 In this context, the action of the holomorphic motion over equilibrium states
 is straightforward.
If $\nu$ is a $f_{\lam_0}$-invariant measure  with support in $J_{\lam_0}$, the measure $\nu_\lam := (\Phi_{\lam_0,\lam})_* \nu$ inherits all the statistical properties from $\nu$.
In particular, if $\nu$ is an equilibrium state with respect to a Hölder continuous weight, since the map $\Phi_{\lam_0,\lam}$ itself is a Hölder continuous conjugacy (see for instance \cite{Duj06}), the measure $\nu_\lam$ is still an equilibrium state with respect to a Hölder continuous weight. We refer to \cite{PU10} as a general reference for the thermodynamical formalism in dimension $1$. \\

The definition and characterization of stability for families in 
any dimension $k\geq 1$ are due to  Berteloot-Bianchi-Dupont  \cite{BBD18, B19},
see also \cite{BB18}
and \cite{BB07, B18, BB23, Pha05} for further characterizations.
As Hurwitz Theorem 
does not hold in higher dimension,
the approach
relies on ergodic and pluripotential techniques.
In particular,
the holomorphic motion as above is replaced
by a more adapted measure-theoretic notion.
Consider the space $\J$ of 
all holomorphic maps $\gamma : M \to \P^k$ such that $\gamma(\lam) \in J_\lam:=\supp \mu_\lam$  for any parameter $\lam \in M$, where $\mu_\lambda$ is the unique ergodic measure of maximal entropy of $f_\lambda$, see \cite{BD01,DS10,FS94}. 
Define the dynamical system $\F : \J \to \J$ induced on $\J$ by
$\F(\gamma)(\lam) := f_\lam(\gamma(\lam))$, for all $\gamma \in \J$ and $\lam \in M$.
For each parameter $\lam \in M$, we 
also
consider the evaluation map $p_\lam : \J \to J_\lam$ defined by $p_\lam(\gamma) := \gamma(\lam)$, where $\gamma\in\mathcal{J}$.

 \begin{definition} \label{d:nu_web} Fix $\lambda_0\in M$. Let $\nu$ be an $f_{\lambda_0}$-invariant probability measure supported on $J_{\lambda_0}$. A
\emph{$(\lambda_0, \nu)$-web}
(or \emph{$\nu$-web} for brevity) is an
$\mathcal F$-invariant compactly supported Borel probability measure $\mathcal{M}$ on $\mathcal{J}$ such that
    $(p_{\lambda_0})_*\mathcal{M}=\nu$.
\end{definition}
 Dynamical stability in the sense of Berteloot-Bianchi-Dupont is equivalent to the existence of a  special (\emph{acritical}, see Remark \ref{r:acritical}) \emph{equilibrium web}, namely a probability measure $\M$ which is a $(\lam, \mu_\lam)$-web for every parameter $\lam \in M$. Such a probability measure
 is also called a
 \emph{measurable holomorphic motion}.
The construction was generalized by
Bianchi and the second author
\cite{BR22}
to a larger class of measures
including, among others,  all measures whose measure-theoretic entropy is larger than $(k-1)\log d$.
This class includes, for instance, all equilibrium states with respect to any H\"older continuous function $\psi$ with $\max \psi - \min \psi < \log d$, whose existence, uniqueness, and properties were recently proved in \cite{BD20,BD22,SUZ,UZ13}, see also \cite{Dup10}. \\
 
In this context, it is a natural question 
whether
the class of equilibrium states is preserved by the measurable holomorphic motions, i.e., if, given
an equilibrium state $\nu$
for $f_{\lam_0}$, the measures $\nu_\lam := (p_\lam)_*\M$ also satisfy the same property, where $\mathcal{M}$ is a $(\lam_0,\nu)$-web.
Observe that, if $f_{\lam_0}$ is hyperbolic (i.e., uniformly expanding on $J_{\lam_0}$), the measurable holomorphic motion is indeed a holomorphic motion in a neighborhood of $\lam_0$, hence gives a
topological conjugacy on the Julia sets
(see for instance \cite{Jon98}). On the other hand, unlike in dimension 1 or in specific families \cite{AB23},
we do not know whether the presence of a hyperbolic parameter in a stable family implies that all parameter are hyperbolic, hence the propagation of the equilibrium states in a stable family is not clear even in the presence of hyperbolic parameters.
More generally, the theory of stability in higher dimensions is known to have a number of different features with respect to the one-dimensional counterpart, see for instance \cite{AB22, BT17, Bie19, Duj17, Taf21}.\\

In this note, we give the following  answer
to the above problem.

\begin{thm} \label{t:main}
Let $\{f_\lam : \P^k \to \P^k\}_{\lam \in M}$ be a stable
family of endomorphisms of $\P^k$ of algebraic degree $d\ge 2$.
Consider a parameter $\lam_0 \in M$ and a Borel function $\psi : \P^k \to \R$ satisfying $\sup \psi - \inf \psi < \log d$.
Assume there exists an 
(respectively  a unique) equilibrium state $\nu$ for the system $(J_{\lam_0},f_{\lam_0})$ with respect to $\psi$
and let
$\M$ be a 
$(\lam_0, \nu)$-web. 
Then
\begin{enumerate}
\item[{\rm (i)}] the $(\lam_0, \nu)$-web $\M$ is unique; 
\item[{\rm (ii)}] for every parameter $\lam \in M$, there exists a Borel function $\psi_\lam : \P^k \to \R$ such that $\nu_\lam := (p_\lam)_*\M$ is a (resp.  the unique) equilibrium state for the system
$(J_\lam,f_\lam)$ with respect to $\psi_\lam$.

\end{enumerate}
\end{thm}

The condition $\sup \psi - \inf \psi < \log d$ implies that the measure-theoretic entropy of $\nu$ satisfies $ h_\nu(f_{\lam_0})>(k-1)\log d$
(hence, in particular, it has strictly positive Lyapunov exponents \cite{dTh08, D12}, and falls within the scope of \cite{BR22}). The existence of  a $(\lam_0,\nu)$-web then follows from \cite[Corollary 1.5]{BR22}. 
The proof of the above result relies
on the existence of a $\mathcal M$-full-measure subset of $\mathcal J$ of elements whose graphs do not intersect
and is given in Section \ref{s:propagation}.
In Section \ref{s:big_system} we also give an interpretation of $\mathcal M$ as an equilibrum
state
for the dynamical system $(\mathcal J, \mathcal F)$, see
Proposition \ref{p:M-eq-state}.
Following this interpretation, it is also natural to study the distribution of repelling graphs with respect to $\mathcal M$, which corresponds to the distribution of repelling points for the systems $(\mathcal J, \mathcal F)$. This is done in \cite{BB23b}.

\subsection*{Acknowledgments}
The authors would like to thank Fabrizio Bianchi for
helpful discussions during the preparation of this paper, and Yûsuke Okuyama and Gabriel Vigny for their 
very useful comments.

\section{Proof of the Main Theorem} \label{s:propagation}

Given a holomorphic family $\{f_\lam : \P^k \to \P^k\}_{\lam \in M}$ of endomorphisms of $\mathbb P^k$,
observe that for every parameter $\lam \in M$, the evaluation map $p_\lam : \J \to J_\lam$ is continuous, and that $p_\lam \circ \F = f_\lam \circ p_\lam$.
The map $p_\lam$ may turn out not to be injective. This is due to the possible
existence of some crossings between graphs in $\J$.
The notion of lamination, introduced in \cite{BBD18}, will be central in our approach.

\begin{definition}
A \emph{lamination} is a  Borel subset $\L \subset \J$ such that for any maps $\gamma_1, \gamma_2 \in \L$, we have
\[\Gamma_{\gamma_1} \cap \Gamma_{\gamma_2} \neq \emptyset \implies \gamma_1 = \gamma_2, \]
where $\Gamma_\gamma \subset M\times \P^k$ denotes the graph of an element $\gamma\in \mathcal{J}$.
If $\M$ is a Borel probability measure on $\J$, we say that $\L$ is an \emph{$\M$-lamination} if, furthermore, $\M(\L) = 1$.
\end{definition}

Given an $\F$-invariant lamination $\L \subset \J$, one can consider the restriction $p_\lam|_\L \: : \L \to \L_\lam$ of the application $p_\lam$, where $\L_\lam := p_\lam(\L)$. The map $p_\lam|_\L$ is automatically $1$-to-$1$. Yet, observe that $p_\lam|_\L^{-1}$ may not be continuous.

\begin{lemma}\label{l:uniqueness_web}
Let $\L \subset \J$ be a lamination, and fix a parameter $\lam \in M$. Then,
for any Borel subset $\A \subset \J$, we have
\[\L \cap \A = \L \cap p_\lam^{-1}(p_\lam(\L \cap \A)).\]
In particular, for any probability measure $\nu$ on $\P^k$, there exists at most one $(\lam,\nu)$-web $\M$ such that
$\M(\L) = 1$.
\end{lemma}
\begin{proof}
Let $\A \subset \J$ be a Borel set.
The direct inclusion is immediate. For the converse one, take $\gamma \in \L$
such that $\gamma(\lam) \in p_\lam(\L \cap \A)$.
Take
$\gamma' \in \L \cap \A$ such that $\gamma(\lam) = \gamma'(\lam)$.
Since $\L$ is a lamination, we have $\gamma = \gamma'$. In particular, we have $\gamma \in \L \cap \A$, which proves the desired equality.\\

Let now $\mathcal M$ be as in the statement. We have
\begin{equation}\label{eq:m(a)=nu(a)}
    \M(\A) = \M(\L \cap \A) = \M(\L \cap p_\lam^{-1}(p_\lam(\L \cap \A))) = \nu(p_\lam(\L \cap \A)).
\end{equation}
The last expression does not depend on the choice of the $(\lam,\nu)$-web $\M$. The proof is complete.
\end{proof}

From now on, we will assume that the family $\{f_\lam : \P^k \to \P^k\}_{\lam \in M}$ is stable.
Our goal is to propagate the class of equilibrium states. To do this, we will make use of the $\M$-laminations.
Given a lamination $\L$,
for every $\lam_0, \lam \in M$ we can consider the bijection $\Phi_{\lam_0,\lam}^{\mathcal L} := p_\lam \circ p_{\lam_0}|_\L^{-1} : \L_{\lam_0} \to \L_\lam$.
\\

\begin{proposition}\label{p:nuweb}
Let $\{f_\lam : \P^k \to \P^k\}_{\lam \in M}$ be a stable 
family of  holomorphic endomorphisms of $\P^k$ of 
algebraic
degree $d \ge 2$. Let $\lam_0 \in M$ and $\nu$ be an ergodic $f_{\lam_0}$-invariant measure with $h_{\nu}(f_{\lam_0}) > (k-1) \log d$. Then
\begin{enumerate}
\item[{\rm (i)}] there exists a unique $(\lam_0, \nu)$-web $\M$. Moreover, $\M$ is ergodic  and for any parameter $\lam \in M$, we have $h_\M(\F) =$ $h_{\nu_\lam}(f_\lam)$, where $\nu_\lam:=(p_\lam)_*\mathcal{M}$;
\item[{\rm (ii)}] if $\L$ is an $\F$-invariant $\M$-lamination and $\psi : J_{\lambda_0}\to \R$ is a Borel measurable function,
then for any parameter $\lambda \in M$, we have 
$P_{\nu_\lam}(\psi_\lam) = P_{\nu}(\psi)$, where $\psi_\lam :=  \psi \circ (\Phi_{\lam_0,\lam}^{\mathcal L})^{-1} \cdot \mathds{1}_{\L_\lam}$.
\end{enumerate}
\end{proposition}

\begin{remark}\label{r:acritical}
The web $\M$ in Proposition \ref{p:nuweb}  is actually \emph{acritical} (see \cite[Definition 1.3]{BR22} and \cite[Theorem 1.4]{BR22}). We do not stress the acriticality in this paper, since the existence of an $\mathcal{M}$-lamination guarantees this condition, and we only work with such laminations.  
\end{remark}

\begin{remark}
If we are allowed to shrink $M$ to some neighbourhood of $\lam_0$ (depending on $\nu$),  
the assumption on $\nu$
can be weakened as in \cite{BR22}, namely it is enough that $\nu$ is a measure on the Julia set with 
strictly positive Lyapunov exponents and not giving  mass to the post-critical set of $f_{\lam_0}$,
see \cite[Theorem 1.4]{BR22}.
The fact that any $\nu$ satisfies these conditions is a consequence of 
\cite{dTh08,D12} and \cite{dTh06,D07}. 
\end{remark}

\begin{proof}[Proof of Proposition \ref{p:nuweb}]

We first prove (i). By \cite[Corollary 1.5]{BR22} there exists an ergodic  $(\lam_0, \nu)$-web $\M$ and an $\F$-invariant $\M$-lamination $\L$.
To show the uniqueness property, consider another $(\lam_0, \nu)$-web $\M'$. We have $\M'(\L) = \nu(\L_{\lam_0}) = \M(\L) = 1$. Lemma \ref{l:uniqueness_web} then yields $\M = \M'$.
Take $\lam \in M$.
The application $p_\lam|_{\L} : \L \to \L_\lam$ is continuous, hence Borel measurable. The space $\J$ being Polish, as a closed subspace of the space $\O(M,\P^k)$ of all holomorphic functions from $M$ to $\P^k$, since $\L\subset \J$ is a Borel subset, the inverse application $p_\lam|_{\L}^{-1}$ is automatically Borel measurable (see for example \cite[Theorem 3 \S 39.V]{Kur66}).
Since $\nu_\lam = (p_\lam)_*\M$, the two measurable systems $(\L, \F, \M)$ and $(\L_\lam, f_\lam, \nu_\lam)$ are then conjugated by $p_\lam$.
The conclusion follows from the fact that
the measure-theoretic entropy is invariant under measurable conjugacy between subsets of full measure.\\

We now prove (ii). Fix $\lam \in M$. The above shows that the measurable isomorphism $\Phi_{\lam_0,\lam}^{\mathcal L}$ conjugates the two systems $(\L_{\lam_0}, f_{\lam_0}, \nu)$ and $(\L_{\lam}, f_{\lam}, \nu_\lam)$, and that $h_{\nu}(f_{\lam_0}) = h_{\nu_\lam}(f_\lam)$. It remains to show that $\int \psi\nu=\int\psi_\lam\nu_\lam$.
Since $(\Phi_{\lam_0,\lam}^{\mathcal L})_*\nu_\lam=\nu$ and $(\Phi_{\lam_0,\lam}^{\mathcal L})^*\psi_\lambda=\psi$ on $\mathcal{L}_{\lam_0}$, we have
\begin{align*}
     \int\psi_\lam\nu_\lam&=\int_{\mathcal L_\lambda}\psi_\lambda\nu_\lam=\int_{\Phi_{\lam_0,\lam}^{\mathcal L}(\mathcal{L}_{\lambda_0})}\psi_\lambda\nu_\lam\\
    &=\int_{\mathcal{L}_{\lambda_0}}(\Phi_{\lam_0,\lam}^{\mathcal L})^*\psi_\lambda\cdot (\Phi_{\lam_0,\lam}^{\mathcal L})^*\nu_\lam=\int_{\mathcal{L}_{\lambda_0}}\psi\nu=\int\psi\nu.
\end{align*}
The proof is complete.
\end{proof}

\begin{remark}
Since the mapping 
$\Phi_{\lam_0,\lam}^{\mathcal L}$ is a measurable conjugacy, the proof of Proposition \ref{p:nuweb} shows that the measure $\nu_\lam$ inherits the statistical properties that $\nu$ may have, 
such as (exponential) mixing,
Central Limit Theorem, Almost Sure Invariant Principle or Large Deviation Principle (see \cite{BD20, BD22, Dup10, SUZ, UZ13} for large classes of examples of measures satisfying these properties). On the other hand, as the space of H\"older functions is in general not preserved by the measurable conjugacy, the space of observables associated with each $\nu_\lam$ will in general depend on $\lambda$.
\end{remark}

\begin{proof}[Proof of the Main Theorem]

Recall that $\mu_\lam$ denotes the measure of maximal entropy of the endomorphism $f_\lam$.
Up to adding to $\psi$ a constant, for simplicity, we will assume that $\inf \psi =0$.
Then, by the definition of the equilibrium state $\nu$, we have
\begin{equation}\label{eq:inequality_1}
h_{\nu}(f_{\lam_0}) + \int\psi\nu = P_{\nu}(\psi) \ge P_{\mu_{\lam_0}}(\psi) = k\log d +  \int\psi \mu_{\lam_0} \ge k \log d.
\end{equation}
By assumption, we have
$\int\psi\nu< \log d$. Hence, \eqref{eq:inequality_1} yields $h_{\nu}(f_{\lam_0}) > (k-1) \log d$.
Proposition \ref{p:nuweb} then yields a unique $(\lam_0,\nu)$-web $\M$, which proves assertion (i).\\

Denote $\nu_\lam = (p_\lam)_* \M$. Observe that Proposition \ref{p:nuweb} also gives that $h_{\nu_\lam}(f_\lam) > (k-1) \log d$ for any parameter $\lam \in M$.
For every parameter $\lam \in M$, consider the Borel function $\psi_\lam : \P^k \to \R$ as defined in Proposition \ref{p:nuweb}.
To prove assertion (ii), consider a parameter $\lam \in M$ and assume that
there exists an ergodic $f_\lam$-invariant measure $\omega_\lam$ such that
\begin{equation}\label{eq:pressure_inequality_1}
P_{\omega_\lam}(\psi_\lam) \ge P_{\nu_\lam}(\psi_\lam).
\end{equation}
Up to replacing $\omega_\lam$ by $\mu_\lam$
if
$P_{\omega_\lam}(\psi_\lam) < P_{\mu_\lam}(\psi_\lam)$,
we can also assume that
$h_{\omega_\lam}(f_\lam) > (k-1)\log d$.

Proposition \ref{p:nuweb} then yields a $(\lam, \omega_\lam)$-web $\M'$ and
we have 
\begin{equation}\label{eq:final_argument_2}
P_{\omega}(\psi) = P_{\omega_\lam}(\psi_\lam) \ge P_{\nu_\lam}(\psi_\lam) = P_{\nu}(\psi),
\end{equation}
where $\omega := (p_{\lam_0})_*\M'$.
 As $\nu$ is an equilibrium state with respect to $\psi$, the inequality in \eqref{eq:final_argument_2}
 is actually an equality. Thus, $\nu_\lam$ is an equilibrium state with respect to $\psi_\lam$.\\
 
 Assume now that $\nu$ is unique. Then, \eqref{eq:final_argument_2} shows that $\omega$ as above is an equilibrium state with respect to $\psi$ as well. By the uniqueness assumption, we have $\omega = \nu$ and both $\M$ satisfying $\M'$ are $(\lam_0, \nu)$-webs and $\M'(\L) = \M'(p_{\lam_0}^{-1}(\L_{\lam_0}) = \omega(\L_{\lam_0}) = \nu(\L_{\lam_0}) = \M(\L) = 1$. Lemma \ref{l:uniqueness_web} then yields $\M = \M'$. In particular, one has $\omega_\lam = \nu_\lam$, which ends the proof.
\end{proof}

\begin{remark}
 As we have seen above, the condition $\sup \psi - \inf \psi < \log d$ guarantees that we have $ h_\nu(f_{\lam_0})>(k-1)\log d$. This,  
 thanks to Proposition \ref{p:nuweb},
 is one of the main ingredients in the proof of the Main Theorem. Let us assume that we are just given
 an equilibrium state $\nu$ with respect to some bounded Borel function $\psi$ 
 (with no further conditions)
 satisfying $ h_\nu(f_{\lam_0})>(k-1)\log d$.  
Then, with the notations as in the proof of the Main Theorem,
Proposition \ref{p:nuweb} still 
implies that $P_{\nu_\lam}(\psi_\lam) = P_{\nu}(\psi)$. On the other hand, in this case, a priori $\nu_\lambda$ may not be an equilibrium state of $\psi_\lambda$. Indeed, following the proof of the Main Theorem,
let $\omega_\lambda$ be an $f_\lambda$-invariant measure satisfying \eqref{eq:pressure_inequality_1}.
 Even if $P_{\omega_\lam}(\psi_\lam) \ge k \log d$, since we do not have any condition on $\psi$, the integral $\int\psi_\lam\omega_\lam$ can be large enough so that 
$ h_{\omega_\lambda}(f_{\lam})\le(k-1)\log d$. 
  Hence, in this case, we cannot apply Proposition \ref{p:nuweb} to obtain \eqref{eq:final_argument_2}.
  \end{remark}

\section{The web $\M$ is an equilibrium state}\label{s:big_system}

In this section, we observe that the system $(\J, \F)$ itself is well-suited for the thermodynamical formalism. To this purpose, an additional compactness assumption is necessary. Define an \emph{equilibrium state} for the system $(\J,\F)$ with respect to a Borel bounded function $\Psi : \J \to \R$ as an ergodic $\F$-invariant compactly supported probability measure $\M$ on $\J$ maximizing the functional $P_\M(\Psi) = h_\M(\F) + \int \Psi \M$.
We have to ensure that the set of ergodic $\F$-invariant compactly supported measures on $\J$ is not empty. In general, one may need to shrink $M$ so that $\mathcal J$ becomes compact. Assuming stability,  that is simply guaranteed by the existence of an
equilibrium web as in \cite{BBD18}.\\

In this framework, given an equilibrium state $\nu$ with respect to a weight function $\psi$ at a parameter $\lam_0$, we construct a naturally defined weight function $\Psi$ on $\J$, compatible with the measurable holomorphic motion, with respect to which the $(\lam_0,\nu)$-web $\M$ from the Main Theorem is an equilibrium state. We show as well that the uniqueness property holds.
In particular, the equilibrium web of Berteloot-Bianchi-Dupont is the unique measure of maximal entropy for the dynamical system $(\J,\F)$.

\begin{proposition}\label{p:M-eq-state}
Under the hypothesis of the Main Theorem, there exists a Borel bounded function $\Psi : \J \to \R$ such that $\Psi = \psi \circ p_{\lam_0}$ $\M$-a.e. and
$P_\M(\Psi) = P_{\nu}(\psi)$. Moreover, the $(\lam_0,\nu)$-web $\M$ is an (resp. the unique) equilibrium state for the system $(\J,\F)$ with respect to $\Psi$.
\end{proposition}

\begin{proof}
As in the proof of the Main Theorem,
up to adding to $\psi$ a constant, for simplicity, we will assume that $\inf \psi =0$. As in the proof of Proposition \ref{p:nuweb}, let $\mathcal{L}$ be a $\mathcal{M}$-lamination 
 (provided by \cite{BR22})
 and define 
$\Psi := \psi \circ p_{\lam_0}|_\L \cdot \mathds{1}_\L$.
We first prove the equality $P_\M(\Psi) = P_{\nu}(\psi_{\lam_0})$. By Proposition \ref{p:nuweb} we have $h_{\M}(\F)=h_{\nu}(f_{\lam_0}).$ So, it is enough to prove that $\int\Psi\M= \int\psi\nu$.
Since $\M(\L)=1$, we have
\[     \int\Psi\M=\int_{\mathcal L}\psi \circ p_{\lam_0}\M=\int_{\mathcal L}(p_{\lam_0})^*(\psi\cdot (p_{\lam_0})_*\M)
=\int_{\mathcal{L}_{\lambda_0}}\psi\nu=h_{\nu}(f_{\lam_0}),
    \]
which gives the desired equality.\\

We now prove the last assertion. Assume there exists another ergodic $\F$-invariant Borel probability measure $\M'$ with compact support on $\J$ such that
$$
P_{\M'}(\Psi) \ge P_{\M}(\Psi).
$$
Define $\omega := (p_{\lam_0})_*\M'$, so that the measure $\M'$ is a $(\lam_0, \omega)$-web.
Observe that Proposition \ref{p:nuweb} yields \begin{equation}\label{eq:first_inequ}
P_{\omega}(\psi) = P_{\M'}(\Psi) \ge P_\M(\Psi) = P_{\nu}(\psi).
\end{equation}
As $\nu$ is an equilibrium state with respect to $\psi$, we have that $P_{\nu}(\psi) \ge P_{\omega}(\psi)$. Hence \eqref{eq:first_inequ} yields $P_{\M'}(\Psi) = P_\M(\Psi)$.
This shows that $\mathcal M$ is an equilibrium state for $\Psi$.
Moreover, if $\nu$ is the unique equilibrium state, then $\omega=\nu$.   Then Proposition \ref{p:nuweb} (i) 
implies that $\M' = \M$. The proof is complete. 
\end{proof}

\bibliographystyle{alpha}

\begin{thebibliography}{9}



\bibitem{AB23} Matthieu Astorg and Fabrizio Bianchi, \emph{Hyperbolicity and Bifurcations in holomorphic families of polynomial skew products}, American Journal of Mathematics,
{\bf 145.3} (2023), 861-898.

\bibitem{AB22} Matthieu Astorg and Fabrizio Bianchi,
\emph{Higher bifurcations for polynomial skew products},
Journal of Modern Dynamics {\bf 18} (2022), 69-99.


\bibitem{BB07}
Giovanni
Bassanelli
and François Berteloot,
\textit{Bifurcation currents in holomorphic dynamics on $\mathbb P^k$},
Journal für die reine und angewandte Mathematik (Crelle's Journal)
{\bf 606} (2007), 201-235.

\bibitem{B18} François Berteloot, \emph{Les cycles r\'epulsifs bifurquent en cha\^{i}ne,} Annali di Matematica Pura ed Applicata, {\bf 197.2} (2018), 517-520. 

\bibitem{BB18} François Berteloot and Fabrizio Bianchi, \emph{Stability and bifurcations in projective holomorphic dynamics},
Dynamical systems, Banach Center Publications, Polish Academy of Sciences Institute of Mathematics, Warsaw, {\bf 115} (2018), 37–71.


\bibitem{BBD18} François Berteloot, Fabrizio Bianchi, and Christophe Dupont, \emph{Dynamical stability and Lyapunov exponents for holomorphic endomorphisms of $\mathbb{P}^k$,} Annales Scientifiques de l'ENS,
{\bf 51.1} (2018), 215-262.

\bibitem{BB23} 
François Berteloot and Maxence Br\'evard, \emph{Ramification current, post-critical normality and stability of
holomorphic endomorphisms of $\mathbb P^k$}, Fundamenta Mathematicae, {\bf 263} (2023), 253-270.



\bibitem{B16} Fabrizio Bianchi, \emph{Motions of Julia sets and dynamical stability in several complex variables}, PhD
Thesis, Université Toulouse III Paul Sabatier and Università di Pisa (2016).

\bibitem{B19} Fabrizio Bianchi, \textit{Misiurewicz parameters and dynamical stability of polynomial-like maps of large
topological degree}, Mathematische Annalen, {\bf 373} (2019), 901–928.

\bibitem{BD20} Fabrizio Bianchi and Tien-Cuong Dinh, Equilibrium states of endomorphisms of $\mathbb P^k$ I: existence and properties,
\emph{Journal de mathématiques pures et appliquées},
{\bf 172} (2023), 164-201.

\bibitem{BD22} Fabrizio Bianchi and Tien-Cuong Dinh, \emph{Equilibrium states of endomorphisms of $\mathbb P^k$: spectral gap and limit theorems}, Geometric and Functional Analysis (GAFA), to appear (2024). 


\bibitem{BB23b} Fabrizio Bianchi and Maxence Brévard, \emph{Holomorphic motions of weighted periodic points}, arXiv preprint arXiv:2307.12734 (2023).


\bibitem{BR22} Fabrizio Bianchi and Karim Rakhimov, \emph{Strong probabilistic
stability in holomorphic families
of endomorphisms of $\P^k$ and polynomial-like maps}, International Mathematics Research Notices, to appear (2024).


\bibitem{BT17} Fabrizio Bianchi and Johan Taflin,
\emph{Bifurcations in the elementary Desboves family}, {\bf 145.10} (2017), 4337-4343.

\bibitem{Bie19} Sébastien Biebler,
\emph{Lattès maps and the interior of the bifurcation locus}, Journal of Modern Dynamics, {\bf 15} (2019), 95-130.

\bibitem{B23} Maxence Br{\'e}vard, \emph{\'Etude des composantes de stabilité pour les familles d'endomorphismes holomorphes des espaces projectifs complexes}, PhD thesis, Université Toulouse III Paul Sabatier (2023).



\bibitem{BD01} Jean-Yves Briend and Julien Duval, \emph{Deux caractérisations de la mesure d'équilibre d'un endomorphisme de $\C\P^k$}, Publications Mathématiques de l'IHÉS, {\bf 93} (2001), 145-159, and \emph{erratum} in {\bf 109} (2009), 295-296.

\bibitem{DeT05} Henry De Thélin,
\emph{Un phénomène de concentration de genre},
Mathematische Annalen {\bf 332} (2005), 483-498.

\bibitem{dTh06}
Henry De Thélin,
Sur la construction de mesures selles,
\emph{Annales de l'Institut Fourier},
{\bf 56.2} (2006), 337-372.

\bibitem{dTh08} Henry De Thélin, \emph{Sur les exposants de Lyapounov des applications m\'eromorphes}, Inventiones Mathematicae, {\bf 172} (2008), 89–116. 

\bibitem{dM03} Laura DeMarco,
\emph{Dynamics of rational maps: Lyapunov exponents, bifurcations, and capacity}, Mathematische Annalen, {\bf 326.1} (2003), 43–73. 

\bibitem{D07} Tien-Cuong Dinh,
\emph{Attracting current and equilibrium measure for attractors on $\P^k$}, Journal of Geometric Analysis, {\bf 17.2} (2007), 227-244.


\bibitem{DS10} Tien-Cuong Dinh and Nessim Sibony, \emph{Dynamics in several complex variables: endomorphisms of projective spaces and polynomial-like mappings},  Holomorphic dynamical systems,
Lecture Notes in Mathematics
1998, Springer, Berlin (2010), 165–294.


\bibitem{Duj06} Romain Dujardin, \emph{Approximation des fonctions lisses sur certaines laminations}, Indiana University Mathematics Journal, {\bf 55} (2006), 579-592.


\bibitem{Duj17}
Romain Dujardin,
\emph{Non density of stability for holomorphic mappings on $\P^k$},
Journal de l'École Polytechnique,
{\bf 4} (2017), 813-843.

\bibitem{D12} Christophe Dupont, \emph{Large entropy measures for endomorphisms of $\mathbb{CP}^k$}, Israel Journal of Mathematics, {\bf 192} (2012), 505–533.

\bibitem{Dup10} Christophe Dupont, \emph{Bernoulli coding map and almost sure invariance principle for endomorphisms of $\P^k$}, Probability theory and related fields {\bf 146.3-4} (2010), 337-359.



\bibitem{FS94} John Erik Fornaess and Nessim Sibony,
Complex dynamics in higher dimensions,
\emph{Complex potential theory (Montreal, PQ, 1993)}, NATO Advanced Science Institute Series C: Mathematical and Physical Sciences, Kluwer Academic Publichers, Dordrecht, {\bf 439} (1994), 131–186.


\bibitem{Jon98} Mattias Jonsson, \emph{Holomorphic motions of hyperbolic sets},
Michigan Mathematical Journal, {\bf 45} (1998), 409-415.

\bibitem{Kur66} Kazimierz Kuratowski, \emph{Topology},
Volume I, Academic Press, New York (1966).


\bibitem{Ly83a} Mikhail Lyubich, \emph{Some typical properties of the dynamics of rational mappings}, Russian Mathematical Surveys, {\bf 38.5} (1983), 154–155.

\bibitem{Ly83b} Mikhail Lyubich,
\emph{Entropy properties of rational endomorphisms of the Riemann sphere},
Ergodic Theory and Dynamical Systems, {\bf 3} (1983), 351-385.

\bibitem{MSS83} Ricardo Mañé, Paulo Sad, and Dennis Sullivan, \emph{On the dynamics of rational maps}, Annales Scientifiques de l'ENS (4), {\bf 16.2} (1983), 193–217.



 \bibitem{Pha05} Ngoc-mai Pham, \emph{Lyapunov exponents and bifurcation current for polynomial-like maps},  Preprint arXiv:0512557 (2005).


\bibitem{Prz85} Feliks Przytycki, \emph{Hausdorff dimension of harmonic measure on the boundary of an attractive basin for a holomorphic map}, Inventiones mathematicae, {\bf 80.1} (1985), 161–179. 


\bibitem{PU10} 
Feliks Przytycki and Mariusz Urbański
\emph{Conformal fractals: ergodic theory methods}, Volume 371, Cambridge University Press (2010).

\bibitem{Si81} Nessim Sibony, Seminar talk at Orsay, October (1981).

\bibitem{SUZ} Micha{\l} Szostakiewicz, Mariusz Urbański, and Anna Zdunik, \emph{Stochastics and thermodynamics for equilibrium measures of holomorphic endomorphisms on complex projective spaces}, Monatshefte für Mathematik, {\bf 174.1} (2014), 141-162.

\bibitem{Taf21}
Johan Taflin,
\emph{Blenders near polynomial product maps of $\C^2$},
Journal of the European Mathematical Society,
{\bf 23} (2021), 3555-3589.

\bibitem{UZ13} Mariusz Urbański and Anna Zdunik, \emph{Equilibrium measures for holomorphic endomorphisms of complex projective spaces}, Fundamenta Mathematicae {\bf 220} (2013), 23-69.

\end{thebibliography}

\end{document}